\documentclass[12pt,oneside,reqno]{amsart}
\usepackage{graphicx}

\graphicspath{ {figs/} }

\usepackage{pict2e}
\usepackage{amssymb}
\usepackage{amsthm}                % better theorem environments
\usepackage[margin=0.9in]{geometry}  % set the margins to 1in on all sides

\usepackage{epstopdf}
\usepackage{amsmath}
\usepackage{graphicx}
\usepackage{subfigure}
\usepackage{latexsym}
\usepackage{amsmath}
\usepackage{amsfonts}
\usepackage{amssymb}
\usepackage{tikz}
\usepackage{mathdots}
\usepackage{comment}
\usepackage{float}
\usepackage[mathscr]{euscript}
\usepackage{enumerate}
\usepackage{epsfig}
\usepackage { graphics }
\usepackage { graphicx }

\usetikzlibrary{arrows.meta}
\usetikzlibrary{knots}
\usetikzlibrary{hobby}
\usetikzlibrary{arrows,decorations.markings}

\usepackage[bookmarks=true,bookmarksnumbered=true, breaklinks=true,pdfstartview=FitH,hyperfigures=false,plainpages=false, naturalnames=true,colorlinks=true,pagebackref=true,pdfpagelabels]{hyperref}

\usepackage[utf8]{inputenc}
\usepackage{longtable}
\usepackage[lite]{amsrefs}
\renewcommand{\PrintDOI}[1]{\href{http://dx.doi.org/\detokenize{#1}}{doi: \detokenize{#1}}%
	\IfEmptyBibField{pages}{, (to appear in print)}{}}

\theoremstyle{definition}
\newtheorem{theorem}{Theorem}[section]
\newtheorem{lemma}[theorem]{Lemma}

\newtheorem{proposition}[theorem]{Proposition}

\theoremstyle{definition}
\newtheorem{definition}[theorem]{Definition}
\newtheorem{example}[theorem]{Example}
\theoremstyle{remark}
\newtheorem{remark}[theorem]{Remark}
\numberwithin{equation}{section}

\numberwithin{equation}{section}
\title{ Polynomial Invariants of Singular Knots and links}

\author{Jose Ceniceros}
\address{Hamilton College, Clinton, NY }
\email{jcenicer@hamilton.edu}

\author{Indu R. Churchill}
\address{State University of New York at Oswego, Oswego, NY }
\email{indurasika.churchill@oswego.edu}

\author{Mohamed Elhamdadi}
\thanks{M.E. was partially supported by Simons Foundation collaboration grant 712462.}
\address{University of South Florida, Tampa, FL }
\email{emohamed@math.usf.edu}

\date{}
\dedicatory{}
\begin{document}
\maketitle 
	
\begin{abstract}
We generalize the notion of the quandle polynomial to the case of singquandles. We show that the singquandle polynomial is an invariant of finite singquandles. We also construct a singular link invariant from the singquandle polynomial and show that this new singular link invariant generalizes the singquandle counting invariant. In particular, using the new polynomial invariant, we can distinguish singular links with the same singquandle counting invariant.

%We apply the invariant by giving examples in which we distinguish between singular knots and links. In particular, the new invariant were able to distinguish two links that have the same psyquandle invariant and Jablan polynomial with different subsingquandle polynomials.  
\end{abstract}

%\tableofcontents

\section{Introduction}

Self-distributive algebraic structures $[(x*y)*z=(x*z)*(y*z)]$, in general, and quandles in particular are non-associative structures motivated by oriented Reidemeister moves in knot theory.  Quandles were introduced independently in \cites{Joyce, Matveev} in the 1980s.  In the past few decades, there have been many studies relating quandles to other areas of mathematics such as the study of the Yang-Baxter equation in \cites{CES, CN}, quasigroups and Moufang loops in \cite{E}, ring theory in \cite{EFT}, representation theory in \cite{EM} and other areas.    
Recently, relations between these distributive structures and singular knot theory have been introduced with the purpose of constructing invariants of singular links in \cites{ BEHY,CCEH, CEHN, NOS}.  A Jones-type invariant for singular links using a variation of the Hecke algebra was constructed and was shown to  satisfy some skein relation in \cite{JL}.  In 1993, connections between Jones type invariants in \cite{Jones} and Vassiliev invariants of singular knots in \cite{V} were established in \cite{BL}.  

In this article, we generalize the notion of the quandle polynomial to the case of singquandles. We also introduce the notion of the fundamental singquandle and use it to define a new polynomial invariant of oriented singular knots and links.  Our invariant generalizes both the coloring invariant of singular links in \cites{BEHY, CEHN} and the polynomial invariant defined in \cite{N}.  We use the new polynomial invariant we construct to distinguish singular links with the same singquandle counting invariant and Jablan polynomial. 

This article is organized as follows.  In Section~\ref{review}, we review the basics of quandles and give some examples. In Section~\ref{quandle polynomial}, we recall the construction of the quandle polynomial in \cites{EN, N}.  In Section~\ref{OSKQ}, we give a diagrammatic definition of singular knots and generalized Reidemeister moves leading to the definition of \emph{oriented singquandles}.  In Section~\ref{Fund}, we define the fundamental singquandle of a singular link and we provide an example.  In Section~\ref{SP}, we introduce a generalization of the quandle polynomial defined in \cite{N} with the aim of defining a polynomial invariant of singular knots and links.  In Section~\ref{examples}, we compute the invariant defined in Section~\ref{SP} and use it %the invariant 
to distinguish singular knots and links.

\section{Basics of Quandles}\label{review}
%We may change this one a bit not to have the same review as our previous paper.
In this section, we review the basics of quandles; more details on the topic can be found in  \cites{EN, Joyce, Matveev}.
\begin{definition}\label{quandle}  A set $(X,\ast)$ is called a \emph{quandle} if the following three identities are satisfied.
\begin{eqnarray*}
& &\mbox{\rm (i) \ }   \mbox{\rm  For all $x \in X$,
$x* x =x$.} \label{axiom1} \\
& & \mbox{\rm (ii) \ }\mbox{\rm For all $y,z \in X$, there is a unique $x \in X$ such that 
$ x*y=z$.} \label{axiom2} \\
& &\mbox{\rm (iii) \ }  
\mbox{\rm For all $x,y,z \in X$, we have
$ (x*y)*z=(x*z)*(y*z). $} \label{axiom3} 
\end{eqnarray*}
\end{definition} 
From Axiom (ii) of Definition~\ref{quandle}, we can write the element $x$ as $z \bar{*} y = x$.  In other words, the equation $x * y=z$ is equivalent to $z \bar{*} y = x$.  The axioms of a quandle correspond respectively to the three Reidemeister moves of types I, II and III (see \cite{EN} for example).  In fact, one of the motivations of defining quandles came from knot diagrammatics. 

 A {\it quandle homomorphism} between two quandles $(X,*)$ and $(Y,\triangleright)$ is a map $f: X \rightarrow Y$ such that $f(x *y)=f(x) \triangleright f(y) $, where
 $*$ and $\triangleright$ 
 denote respectively the quandle operations of $X$ and $Y.$  Furthermore if $f$ is a bijection then it is called a 
{\it quandle isomorphism} between $X$ and $Y$.  %The set of quandle automorphisms of a quandle $X$ is a group denoted by Aut($X$).\\
 \\
% add some review on biquandles.

\noindent The following are some typical examples of quandles.
\begin{itemize}
\item
Any non-empty set $X$ with the operation $x*y=x$ for all $x,y \in X$ is
a quandle called a  {\it trivial} quandle.
\item
Any group $X=G$ with conjugation $x*y=y^{-1} xy$ is a quandle.

\item
Let $G$ be an abelian group.
For elements  
$x,y \in G$, 
define
$x*y \equiv 2y-x$.
Then $\ast$ defines a quandle
structure on $G$ called \emph{Takasaki} quandle.  In case $G=\mathbb{Z}_n$ (integers mod $n$) the quandle is called {\it dihedral quandle}.
This quandle can be identified with  the
set of reflections of a regular $n$-gon
  with conjugation
as the quandle operation.
\item
Any $\Lambda = (\mathbb{Z }[T^{\pm 1}])$-module $M$
is a quandle with
$x*y=Tx+(1-T)y$, $x,y \in M$, called an {\it  Alexander  quandle}.

\item
A {\it generalized Alexander quandle} is defined  by %also regarded as 
a pair $(G, f)$ where 
$G$ is a  group and $f \in {\rm Aut}(G)$,
and the quandle operation is defined by 
$x*y=f(xy^{-1}) y $.  %The previous example corresponds to when $G$ is abelian group. 
\end{itemize}

%\textcolor{red}{Given a quandle $(X,*)$, axiom (ii) states that {\it right multiplication} by $y \in X$, given by by ${\mathcal{R}}_y(x) = x*y$ for $x \in X$ is a permutation of $X$. The subgroup of ${\rm Aut}(X)$, generated by the permutations ${\mathcal{R}}_y$, $y \in X$, is called the {\it inner automorphism group} of $X$,  and is denoted by ${\rm Inn}(X)$. A quandle is {\it connected} if ${\rm Inn}(X)$ acts transitively on $X$. The operation $\bar{*}$ on $X$ defined by $x\ \bar{*}\ y= {\mathcal{R}}_y^{-1} (x) $is a quandle operation.}%, and $(X,  \bar{*}) $ is called the {\it dual} quandle of $(X, *)$.

\section{Review of the Quandle Polynomial}\label{quandle polynomial}
In this section, we recall the definition of the quandle polynomial, the subquandle polynomial, and the link invariants obtained from the subquandle polynomial. For a detailed construction of these polynomials, see \cites{EN,N}.
\begin{definition}
Let  $(Q,*)$ be a finite quandle. For any element $x \in Q$,  let
\[ C(x) = \lbrace y \in Q \, : \, y *x = y \rbrace \quad \text{and} \quad R(x) = \lbrace y \in Q \, : \, x *y = x \rbrace  \]
and set $r(x) = \vert R(x) \vert$ and $c(x) = \vert C(x) \vert$. Then the \emph{quandle polynomial of Q}, $qp_Q(s,t)$, is 
\[ qp_Q(s,t) = \sum_{x \in Q} s^{r(x)}t^{c(x)}.\]
\end{definition}
%Let $(Q',*')$ be a quandle. An isomorphism $\phi: Q \rightarrow Q'$ induces bijections $\phi_r: R(x) \rightarrow R(\phi(x))$ and $\phi_c : C(x) \rightarrow C(\phi(x))$, so that $qp_Q(Q) = qp_Q(Q')$ and $qp_Q$ is an invariant of isomorphism type for finite quandle.

In \cite{N}, the quandle polynomial was shown to be an effective invariant of finite quandles. In addition to being an invariant of finite quandles, the quandle polynomial was generalized to give information about how a subquandle is embedded in a larger quandle.

\begin{definition}
Let $S \subset Q$ be a subquandle of $Q$. The \emph{subquandle polynomial of S}, $qp_{S \subset Q}(s,t)$, is
\[qp_{S \subset Q}(s,t) = \sum_{x \in S} s^{r(x)}t^{c(x)}\]
where $r(x)$ and $c(x)$ are defined above.
\end{definition}
Note that for any knot or link $K$, there is an associated fundamental quandle, $Q(K)$, and for any given finite quandle $T$ the set of quandle homomorphism, denoted by $\textup{Hom}(Q(K),T)$, has been used to define computable link invariants, for example, the cardinality of the set is known as the \emph{quandle counting invariant}. In \cite{N}, the subquandle polynomial of the image of each homomorphism was used to enhance the counting invariant.

\begin{definition}
Let $K$ be a link and $T$ be a finite quandle. Then for every $f \in \textup{Hom}(Q(K),T)$, the image of $f$ is a subquandle of $T$. The \emph{subquandle polynomial invariant}, $\Phi_{qp}(K,T)$, is the set with multiplicities
\[ \Phi_{qp}(K,T) = \lbrace qp_{\textup{Im}(f) \subset T} (s,t) \, \vert \, f \in \textup{Hom}(Q(K),T) \rbrace.\]
Alternatively, the multiset can be represented in polynomial form by
\[  \phi_{qp}(K,T) = \sum_{f \in \textup{Hom}(Q(K),T)} u^{qp_{\textup{Im}(f) \subset T} (s,t)}.\]
\end{definition}

\section{Oriented Singular Knots and Quandles}\label{OSKQ}

%Similarly to biquandles, singular biquandles are algebraic structures that can be derived from the generating set of singular Reidemeister moves given in Figure \ref{Rmoves}. Generating sets of oriented singular Reidemeister moves were studied in \cite{BEHY}. The generating set obtained in \cite{BEHY}  is illustrated in Figure~\ref{generatingset}.

In this section, we review the definition of a singquandle and we define the fundamental singquandle of a singular link. In \cite{BEHY}, generating sets of the generalized Reidemeister moves were studied and were used to define oriented singquandles.  An oriented singquandle can be thought of as an algebraic structure derived from a generating set of the oriented singular Reidemeister moves. We will follow the convention at classical crossings and singular crossings found in \cite{BEHY}, see Figure~\ref{Rmoves}.
\begin{figure}[h]
	\tiny{
		\centering
	{\includegraphics[scale=0.9]{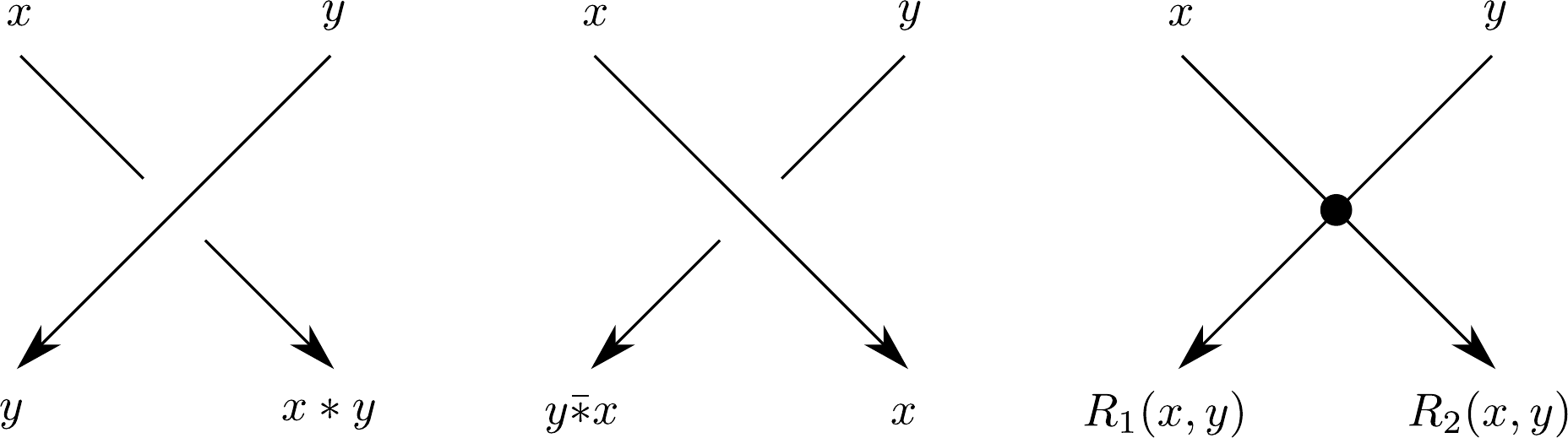}}
		\vspace{.2in}
		\caption{Colorings of classical and singular crossings.}
		\label{Rmoves}}
\end{figure}

The axioms of an oriented singquandle are derived by considering the moves in Figure~\ref{generatingset} and writing the equations obtained using the above binary operations; see Figures \ref{The generalized Reidemeister move RIV}, \ref{The generalized Reidemeister move RIVb} and \ref{The generalized Reidemeister move RV}. Note that by letting $(X,*)$ be a quandle we can ignore the contributions from the classical Reidemeister moves.

%\begin{figure}[h]
\begin{figure}[h] 
%\begin{figure}[htb] 
\tiny{
\centering
    \includegraphics[scale=.8]{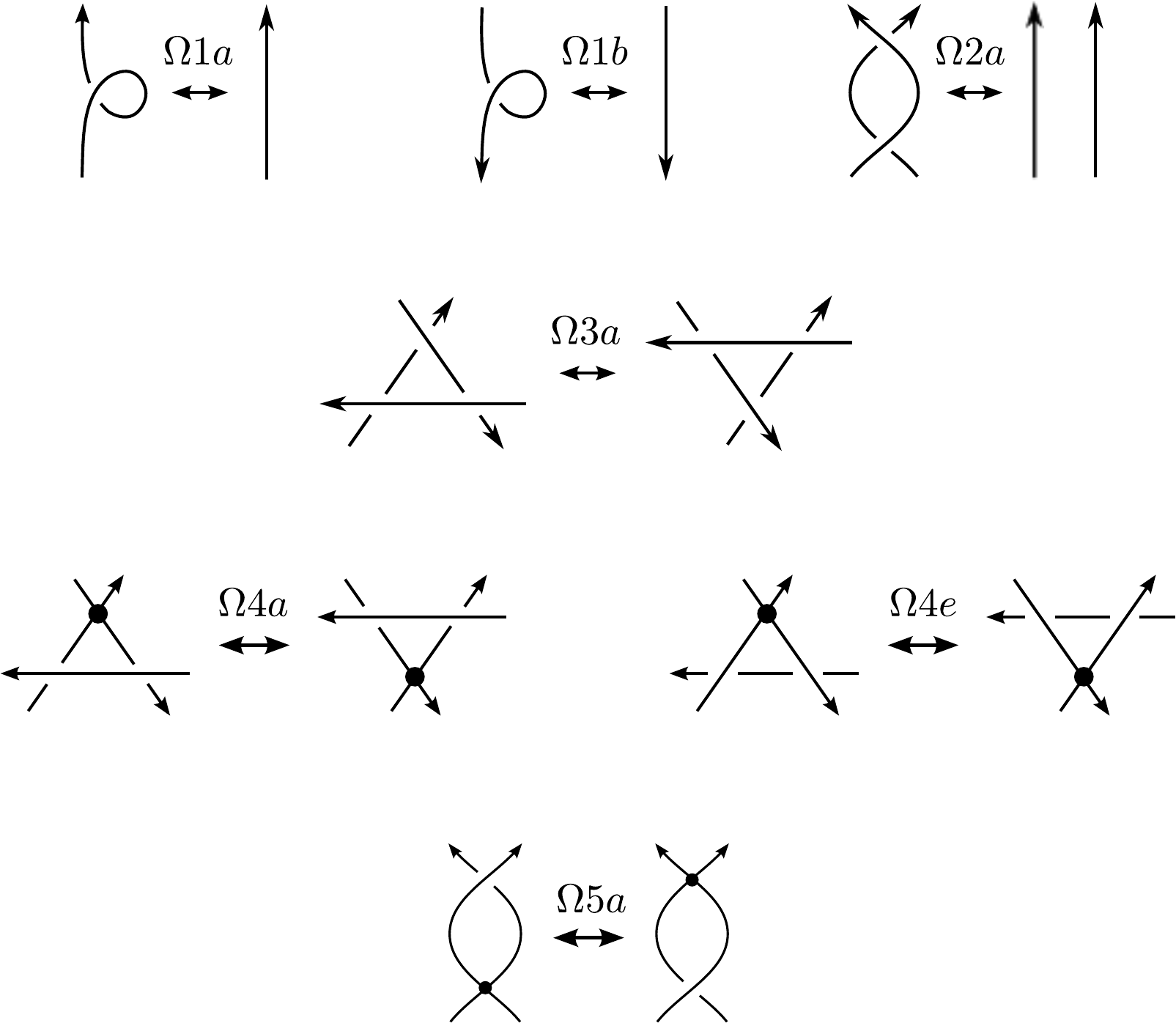}
    \caption{A generating set of singular Reidemeister moves}
    \label{generatingset}}
\end{figure}

\begin{definition}\label{oriented SingQdle}
	Let $(X, *)$ be a quandle.  Let $R_1$ and $R_2$ be two maps from $X \times X$ to $X$.  The quadruple $(X, *, R_1, R_2)$ is called an {\it oriented singquandle} if the following axioms are satisfied for all $a,b,c \in X$:
	\begin{eqnarray}
		R_1(a\bar{*}b,c)*b&=&R_1(a,c*b)  \label{eq1}\\
		R_2(a\bar{*}b, c) & =&  R_2(a,c*b)\bar{*}b \label{eq2}\\
	      (b\bar{*}R_1(a,c))*a   &=& (b*R_2(a,c))\bar{*}c  \label{eq3}\\
R_2(a,b)&=&R_1(b,a*b)   \label{eq4}\\
R_1(a,b)*R_2(a,b)&=&R_2(b,a*b).   \label{eq5}	
\end{eqnarray}	
\end{definition}

%First diagram for p/sIII
\begin{figure}[h] 
\begin{tikzpicture}[use Hobby shortcut, scale=.8]
%diagram on the left
 	\draw[thick,-stealth] (-1,2) .. (2,-2);
	\draw[thick,-stealth] (2,2) ..  (-1,-2);
	\draw [line width=2mm,white,-stealth](.7,-2)..(0,-1.5).. (-1,0)..(0,1.5).. (.7,2);
	\draw [thick,-stealth] (.7,-2)..(0,-1.5).. (-1,0)..(0,1.5).. (.7,2);

	\node[left] at (-1,2) {\tiny $a$};
	\node[right] at (0,.8){\tiny$a\bar{*}b$};
    %\node[right] at (-.7,-.3) {\tiny $R_1(a\bar{*}b,c)$};
    \node[left] at (-1,-2) {\tiny $R_1(a\bar{*}b,c)*b$};
    
    \node[right] at (2,2) {\tiny $c$};
    \node[left] at (.5,-2) {\tiny$b$};
    \node[right] at (2,-2) {\tiny $R_2(a \bar{*}b,c)$};
    \node[circle,draw=black, fill=black, inner sep=0pt,minimum size=6pt] (a) at (.5,0) {};
%Equivalence arrow
    \draw [very thick, <->] (3,0) -- (4,0);

%diagram on the right
	\draw[thick,-stealth] (8,2) .. (5,-2);   
    \draw[thick,-stealth] (5,2).. (8,-2);
    \draw [line width=2mm,white,-stealth] (6.3,-2).. (7,-1.5).. (8,0).. (7,1.5)..(6.3,2);
    \draw[thick,-stealth]    (6.3,-2).. (7,-1.5).. (8,0).. (7,1.5)..(6.3,2);
    
    \node[left] at (5,2) {\tiny $a$};
    \node[left] at (8,2) {\tiny $c$};
    \node[left] at (6.3,-2) {\tiny $b$};

    \node[left] at (7,.8) {\tiny$c*b$};
    
        \node[left] at (5.5,-1.4) {\tiny $R_1(a,c*b)$};
    \node[right] at (8,-2) {\tiny $R_2(a,c*b)\bar{*}b$};
    \node[circle,draw=black, fill=black, inner sep=0pt,minimum size=6pt] (a) at (6.5,0) {};
\end{tikzpicture}
\vspace{.2in}
		\caption{The Reidemeister move $\Omega 4a$ and colorings}
		\label{The generalized Reidemeister move RIV}
\end{figure}
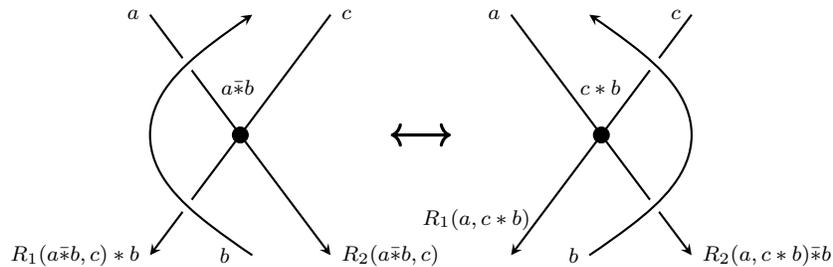
%Second diagram for p/sIII

\begin{figure}[h]
\begin{tikzpicture}[use Hobby shortcut, scale=.8]
%diagram on the left

    \draw [thick,-stealth] (.7,-2)..(0,-1.5).. (-1,0)..(0,1.5).. (.7,2);
    \draw [line width=2mm,white,-stealth](-1,2)..(2,-2);
    \draw [line width=2mm,white,-stealth](2,2)..(-1,-2);
 	\draw[thick,-stealth] (-1,2) .. (2,-2);
	\draw[thick,-stealth] (2,2) ..  (-1,-2);

	\node[left] at (-1,2) {\tiny $a$};
	\node[left] at (-1,0){\tiny$b \bar{*} R_1(a,c)$};
    %\node[right] at (-.7,-.3) {\tiny $R_1(a\bar{*}b,c)$};
    \node[left] at (-1,-2) {\tiny $R_1(a,c)$};
    \node[above] at (.7,2) {\tiny$(b \bar{*} R_1(a,c))*a$};
    \node[right] at (2,2) {\tiny $c$};
    \node[left] at (.5,-2) {\tiny$b$};
    \node[right] at (2,-2) {\tiny $R_2(a,c)$};
    \node[circle,draw=black, fill=black, inner sep=0pt,minimum size=6pt] (a) at (.5,0) {};
%Equivalence arrow
    \draw [very thick, <->] (3,0) -- (4,0);

%diagram on the right
\draw[thick,-stealth]    (6.3,-2).. (7,-1.5).. (8,0).. (7,1.5)..(6.3,2);
\draw [line width=2mm,white,-stealth] (8,2) .. (5,-2);
\draw [line width=2mm,white,-stealth] (5,2).. (8,-2);
	\draw[thick,-stealth] (8,2) .. (5,-2);   
    \draw[thick,-stealth] (5,2).. (8,-2);

    \node[left] at (5,2) {\tiny $a$};
    \node[left] at (8,2) {\tiny $c$};
    \node[left] at (6.3,-2) {\tiny $b$};

    \node[right] at (8,0) {\tiny$b*R_2(a,c)$};
    \node[left] at (5.5,-1.4) {\tiny $R_1(a,c)$};
    \node[right] at (8,-2) {\tiny $R_2(a,c)$};
    \node[above] at (6.3,2) {\tiny $(b*R_2(a,c))\bar{*}c$};
    \node[circle,draw=black, fill=black, inner sep=0pt,minimum size=6pt] (a) at (6.5,0) {};
\end{tikzpicture}
\vspace{.2in}
		\caption{The Reidemeister move $\Omega 4e$ and colorings}
		\label{The generalized Reidemeister move RIVb}
\end{figure}
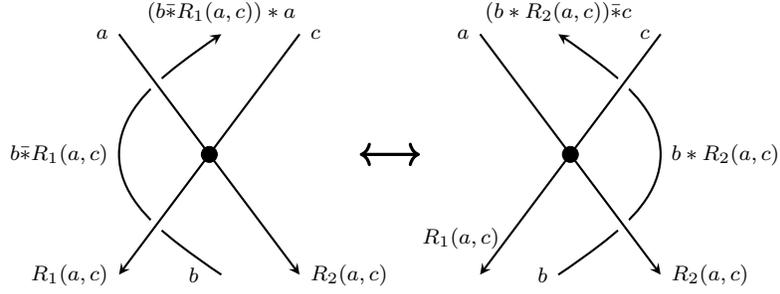

\begin{figure}[h]
\begin{tikzpicture}[use Hobby shortcut, scale=.8]
%diagram on the left
	\draw[thick,-stealth] (-6,2) .. (-6,1.8).. (-7,.2) .. (-7,0)..(-7,-.2) .. (-6,-1.8).. (-6,-2);
	\draw [line width=2mm,white,-stealth](-6,0)..(-6,-.2).. (-7,-1.8)..(-7,-2);	
 	\draw[thick,-stealth] (-7,2).. (-7,1.8).. (-6,.2).. (-6,0)..(-6,-.2).. (-7,-1.8)..(-7,-2) ;
	\node[left] at (-7,2) {\tiny $a$}; 
    \node[left] at (-7,0) {\tiny $R_1(a,b)$};
    \node[left] at (-7,-2) {\tiny $R_2(a,b)$};
	\node[right] at (-6,2) {\tiny $b$};
	\node[right] at (-6,0) {\tiny $R_2(a,b)$};
	\node[right] at (-6,-2) {\tiny $R_1(a,b)*R_2(a,b)$};
\node[circle,draw=black, fill=black, inner sep=0pt,minimum size=6pt] (a) at (-6.5,1) {};
	
%Equivalence arrow
    \draw [very thick, <->] (-3,0) -- (-1,0);

%diagram on the right
	\draw[thick,-stealth] (2,2) .. (2,1.8).. (3,.2) .. (3,0)..(3,-.2) .. (2,-1.8).. (2,-2);
	\draw [line width=2mm,white,-stealth](3,2).. (3,1.8).. (2,.2).. (2,0) ;
	\draw[thick,-stealth] (3,2).. (3,1.8).. (2,.2).. (2,0)..(2,-.2).. (3,-1.8)..(3,-2);
	\node[left] at (2,2) {\tiny $a$}; 
	\node[left] at (2,0) {\tiny $b$};
	\node[left] at (2,-2) {\tiny $R_1(b,a*b)$};
	\node[right] at (3,2) {\tiny $b$};
	\node[right] at (3,0) {\tiny $a * b$};
	\node[right] at (3,-2) {\tiny $R_2(b,a*b)$};
	\node[circle,draw=black, fill=black, inner sep=0pt,minimum size=6pt] (a) at (2.5,-1) {};
\end{tikzpicture}
\vspace{.2in}
		\caption{The Reidemeister move $\Omega 5a$ and colorings}
		\label{The generalized Reidemeister move RV}
\end{figure}
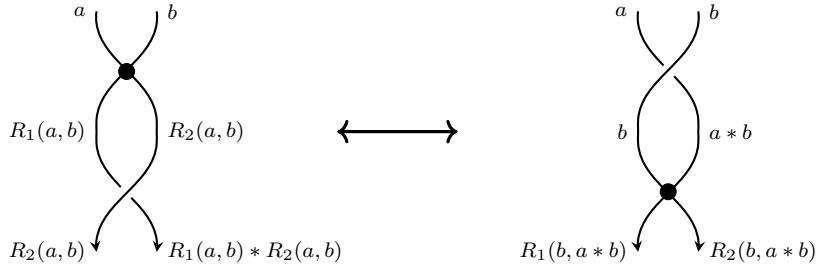

Using affine functions for the maps $*$, $R_1$ and $R_2$ over the set $\mathbb{Z}_n$ of integers mod $n$, the following example was obtained in \cite{BEHY}.
\begin{example}\label{alex}
Take $X=\mathbb{Z}_n$.  For an invertible element $t \in \mathbb{Z}_n$ and any $s \in \mathbb{Z}_n$, define $a*b=ta+(1-t)b$, $R_1(a,b)=sa+(1-s)b$ and $R_2(a,b)=t(1-s)a +(1-t+st)b$.  Then $(\mathbb{Z}_n,*,R_1,R_2)$ is an oriented singquandle.
\end{example}

Now we define the notion of homomorphism and isomorphism of oriented singquandles.  A map $f: X \rightarrow Y$ is called a \emph{homomorphism} of oriented singquandles $(X, *, R_1, R_2)$ and $(Y, \triangleright, R'_1, R'_2)$ if the following conditions are satisfied for all $x,y \in X$
\begin{eqnarray}
f(x*y)&=&f(x) \triangleright f(y)\label{3.6}\\
f(R_1(x,y))&=&R'_1(f(x),f(y))\label{3.7}\\
f(R_2(x,y))&=&R'_2(f(x),f(y)).\label{3.8}
\end{eqnarray}
An oriented singquandle \emph{isomorphism} is a bijective oriented singquandle homomorphism, and two oriented singquandles are \emph{isomorphic} if there exists an oriented singquandle isomorphism between them.
\begin{definition}
Let $(X, *, R_1, R_2)$ be a singquandle.  A subset $M \subset X$ is called a subsingquandle if $(M, *, R_1, R_2)$ is itself a singquandle.  In particular, $M$ is closed under the operations $*, R_1$ and $R_2$.  
\end{definition}

\begin{example}
The quadruple $(\mathbb{Z}, *, R_1,R_2)$ forms a singquandle with $x*y=2y-x$, $R_1(x,y)=\alpha x + (1-\alpha) y$ and $R_2(x,y)=(\alpha -1) x + (2-\alpha) y$, where $\alpha  \in \mathbb{Z}$. 
For a fixed positive integer $n$, consider the subset $n\mathbb{Z} \subset \mathbb{Z}$.  The linearity of the maps $*,R_1$ and $R_2$ immediately implies that $(n\mathbb{Z}, *, R_1,R_2)$ is a subsingquandle of $(\mathbb{Z}, *, R_1,R_2)$.
\end{example}

Given an oriented singquandle homomorphism, we obtain the following lemma.

\begin{lemma}\label{image}
The image, $Im(f)$, of any oriented singquandle homomorphism  $f:(X, *, R_1, R_2) \rightarrow (Y, \triangleright, R'_1, R'_2)$ is always a subsingquandle.
\end{lemma}

\begin{proof}
The equations (\ref{3.6}), (\ref{3.7}) and (\ref{3.8}) imply that $Im(f)$ is closed under $\triangleright, R'_1$ and $R'_2$.  Since the identities of Definition \ref{oriented SingQdle} are satisfied in $Y$ then they are automatically satisfied in $Im(f)$.  This concludes the proof.
\end{proof}

%{\color{red} PLEASE PROOFREAD THIS PART about the Fundamental Singquandle of a Knot. $SQ(K)$.}
\section{Fundamental Singquandle of a Singular Link}\label{Fund}
In this section, we define the fundamental singquandle of a singular link and we provide an example.  Let $D$ represent a diagram of an oriented singular link $L$. We define the \emph{fundamental singquandle of $D$}, denoted by $\mathcal{SQ}(D)$, by proceeding as in classical knot theory with the fundamental quandle.  That is, it is the quotient of the free singquandle on the set $S=\{a_1, \ldots, a_m\}$ of arcs of $D$ at regular crossings and semi-arcs at singular crossings of the diagram $D$ modulo the equivalence relation generated by the regular and singular crossings in the diagram.  
%{\color{blue}In other words, if $x,y \in S$ then $x*y, x\bar{*}y, R_1(x,y), R_2(x,y) \in S$, thus making $S$ a recursive set of generators.  The set of relations come from the conditions at regular and singular crossings.} %{\color{red} 
In other words, if $x,y \in S$ then $x*y, x\bar{*}y, R_1(x,y), R_2(x,y) \in S$, thus making $S$ a recursive set of generators and the set of relations come from the conditions at regular and singular crossings of $D$.  The \emph{fundamental singquandle of $L$}, denoted by $\mathcal{SQ}(L)$, is well defined since any two diagrams $D_1$ and $D_2$  of $L$ are related by the generalized Reidemeister moves (as above) which correspond to making the equivalence relation representing the axioms of Definition~\ref{oriented SingQdle}.  The following is an example for illustration.

\begin{example}
In this example, we present the fundamental singquandle of the following singular link.  Consider the singular Hopf link with one singular crossing and one positive regular crossing represented by diagram $D$.  Following the naming convention in \cite{Oyamaguchi}, we will refer to this singular link as $1_1^l$.
\begin{figure}[h]
\begin{tikzpicture}[use Hobby shortcut]
%diagram on the left
\begin{knot}[
%  draft mode=crossings,
clip width=5,
 % flip crossing=4
]

\strand[decoration={markings,mark=at position .5 with
    {\arrow[scale=3,>=stealth]{<}}},postaction={decorate}](-1,0) circle[radius=2cm];
\strand[decoration={markings,mark=at position .5 with
    {\arrow[scale=3,>=stealth]{>}}},postaction={decorate}] (1,0) circle[radius=2cm];

\end{knot}

\node[circle,draw=black, fill=black, inner sep=0pt,minimum size=12pt] (a) at (0,1.7) {};
\node[left] at (-1,0) {\tiny $z$};
%\node[right] at (1,0) {\tiny $w$};
\node[left] at (-2,2) {\tiny $x$};
\node[right] at (2,2) {\tiny $y$};
\end{tikzpicture}
\vspace{.2in}
		\caption{Diagram $D$ of $1_1^l$.}
		\label{hopf}
\end{figure}
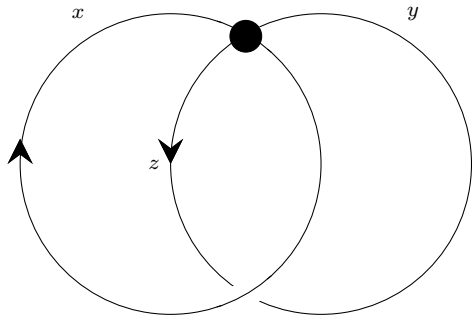
We will label the arcs of the diagram $D$ by $x,y,$ and $z$.  Then the fundamental singquandle of $1_1^l$ is given by 

\begin{eqnarray*}
\mathcal{SQ}(1_1^l)&=&\langle x,y,z; \; x=R_2(x,y), \; z=R_1(x,y), \;z*x=y\rangle\\
&=&\langle x,y; R_1(x,y)*R_2(x,y)=y \rangle.
\end{eqnarray*}
\end{example}

%\textcolor{red}{I though if we take positive regular crossing as mentioned then it would be $y=z * w$ instead $w * z$ isn't it?}

\section{The Singquandle Polynomial and a Singular Link Invariant}\label{SP}

We will define a generalization of the quandle polynomial, introduced in \cite{N}, to define a polynomial invariant of singular knots and links. We will first extend the definition of the quandle polynomial to obtain a singquandle polynomial. The quandle polynomial was extended to biquandles in \cite{N2}; we will similarly extend the quandle polynomial to singquandles. Since we are only considering oriented singquandles, we will refer to them as singquandles.

\begin{definition}
Let $(X,*,R_1,R_2)$ be a finite singquandle. For every $x \in X$, define
\[ C^1(x) = \lbrace y \in X \, \vert \, y * x = y \rbrace \quad \text{and} \quad R^1(x) = \lbrace y \in X \, \vert \, x * y = x \rbrace, \]
\[ C^2(x) = \lbrace y \in X \, \vert \, R_1(y , x) = y \rbrace \quad \text{and} \quad R^2(x) = \lbrace y \in X \, \vert \, R_1(x , y) = x \rbrace, \]
\[ C^3(x) = \lbrace y \in X \, \vert \, R_2(y , x) = y \rbrace \quad \text{and} \quad R^3(x) = \lbrace y \in X \, \vert \, R_2(x , y) = x \rbrace. \]
Let $c^i(x) = \vert C^i(x)\vert$ and $r^i(x) = \vert R^i(x)\vert$ for $i=1,2,3$. Then the \emph{singquandle polynomial of X} is 
\[ sqp(X) = \sum_{x\in X} s_1^{r^1(x)}t_1^{c^1(x)}s_2^{r^2(x)}t_2^{c^2(x)}s_3^{r^3(x)}t_3^{c^3(x)}.  \]
\end{definition}
We note that the value $r^i(x)$ is the number of elements in $X$ that act trivially on $x$, while $c^i(x)$ is the number of elements of $X$ on which $x$ acts trivially.

\begin{proposition}\label{singinv}
If $(X, *, R_1, R_2)$ and $(Y, \triangleright, R'_1, R'_2)$ are isomorphic finite singquandles, then $sqp(X) = sqp(Y)$.
\end{proposition}

\begin{proof}
Suppose $f:X \rightarrow Y$ is a singquandle isomorphism and fix $x \in X$. For all $y \in C^1(x) =\lbrace y \in X \, \vert \, y * x = y \rbrace$, we have $f(y) \triangleright f(x) = f(y*x) = f(y)$, therefore, $f(y) \in C^1(f(x))$ and $\vert C^1(x) \vert \leq \vert C^1(f(x)) \vert$. Applying the same argument to $f^{-1}$, we obtain the opposite inequality, and we have $c^1(x) = c^1(f(x))$. Note that by definition of a singquandle isomorphism we have, $R_1'(f(y),f(x)) = f(R_1(y,x))= f(y) $ and $R_2'(f(y),(x))=f(R_2(y,x)) =f(y)$. Therefore, by applying a similar argument used to show $c^1(x) = c^1(f(x))$ we obtain $c^i(x) =c^i(f(x))$ for $i=2,3$. A similar argument also shows that $r^i(x) = r^i(f(x))$ for $i=1,2,3$. Thus,
\begin{eqnarray*}
sqp(X) &=& \sum_{x\in X} s_1^{r^1(x)}t_1^{c^1(x)}s_2^{r^2(x)}t_2^{c^2(x)}s_3^{r^3(x)}t_3^{c^3(x)}\\
&=&\sum_{f(x)\in Y} s_1^{r^1(x)}t_1^{c^1(x)}s_2^{r^2(x)}t_2^{c^2(x)}s_3^{r^3(x)}t_3^{c^3(x)} \\
&=& \sum_{f(x)\in Y} s_1^{r^1(f(x))}t_1^{c^1(f(x))}s_2^{r^2(f(x))}t_2^{c^2(f(x))}s_3^{r^3(f(x))}t_3^{c^3(f(x))}\\
&=& sqp(Y).
\end{eqnarray*}
\end{proof}

%As for the quandle polynomial case, if $\phi: X \rightarrow X'$ is a singquandle isomorphism induces bijections $\phi_r^{i}: R^{i}(x) \rightarrow R^{i}(\phi(x))$ and $\phi_c^{i}: C^{i}(x) \rightarrow C^{i}(\phi(x))$ for $i = 1,2,3$, so $X \cong X'$ implies $sqp(X) = sqp(X')$ and each $sqp(X)$ is an invariant of %singquandle isomorphism
In the following example we will see that we can compute $r^i(x)$ (or $c^i(x)$) from each operation table by simply going through the rows (columns) and counting the occurrences of the row number.
\begin{example}\label{sqp}
Let $X =\mathbb{Z}_4$ be the singquandle with operations $x*y = 3x+2y $, $R_1(x,y)=2x+3y$, and $R_2(x,y) = x$. We obtain this sinquandle from Example~\ref{alex}, by letting $n=4$, $t=3$, and $s=2$. These operations have the following operation tables,
\[
\begin{tabular}{ r|  c  c  c c }
* & 1 &2 &3 & 0\\
  \hline			
1& 1 & 3 & 1 & 3 \\
2& 0 & 2 & 0 & 2 \\
3& 3 & 1 & 3 & 1 \\
0& 2 & 0 & 2 & 0 \\
\end{tabular}
\qquad
\begin{tabular}{ r|  c  c  c c }
$R_1$ & 1 &2 &3 & 0\\
  \hline			
1&1 & 0 & 3 & 2 \\
2& 3 & 2 & 1 & 0 \\
3& 1 & 0 & 3 & 2 \\
0& 3 & 2 & 1 & 0 \\
\end{tabular}
\qquad
\begin{tabular}{ r|  c  c  c c }
$R_2$ & 1 &2 &3 & 0\\
  \hline			
 1&1 & 1 & 1 & 1 \\
 2&2 & 2 & 2 & 2 \\
 3&3 & 3 & 3 & 3 \\
 0&0 & 0 & 0 & 0 \\
\end{tabular}
\]
and the operations have the following $r^i(x)$ and $c^i(x)$ values for $i= 1,2,3$:
\[
\begin{tabular}{ r|  c  c  }
$x$ & $r^1(x)$ & $c^1(x)$\\
  \hline			
1&  2 & 2 \\
2&  2 & 2  \\
3&  2 & 2 \\
0&  2 & 2  \\
\end{tabular}
\qquad
\begin{tabular}{ r|  c  c  }
$x$ & $r^2(x)$ & $c^2(x)$\\
  \hline			
1&  1 & 1 \\
2&  1 & 1  \\
3&  1 & 1 \\
0&  1 & 1  \\
\end{tabular}
\qquad
\begin{tabular}{ r|  c  c  }
$x$ & $r^3(x)$ & $c^3(x)$\\
  \hline			
1&  4& 4 \\
2&  4 & 4  \\
3&  4 & 4 \\
0&  4 & 4  \\
\end{tabular}.
\]
Thus, the singquandle polynomial of $X$ is 
\[ sqp(X) = 4 s_1^2 t_1^2 s_2 t_2 s_3^4 t_3^4. \]
\end{example}

\begin{example}
Let $Y =\mathbb{Z}_4$ be the singquandle with operations $x \triangleright y =3 x + 2 xy$, $R'_1(x,y) = 3 x + 2  xy + 2 y$, and $R'_2(x,y)=2 x + 2  xy + 3 y$. A direct computation shows that in fact the operations $ \triangleright, R'_1$ and $R'_2$ satisfy the axioms of Definition~\ref{oriented SingQdle}. These operations have the following operation tables,
\[
\begin{tabular}{ r|  c  c  c c }
$\triangleright$ & 1 &2 &3 & 0\\
  \hline			
1&  1 & 3 & 1 & 3 \\
2& 2 & 2 & 2 & 2 \\
3& 3 & 1 & 3 & 1 \\
0& 0 & 0 & 0 & 0 \\
\end{tabular}
\qquad
\begin{tabular}{ r|  c  c  c c }
$R_1'$ & 1 &2 &3 & 0\\
  \hline			
1& 3 & 3 & 3 & 3 \\
2& 0 & 2 & 0 & 2 \\
3& 1 & 1 & 1 & 1 \\
0& 2 & 0 & 2 & 0 \\
\end{tabular}
\qquad
\begin{tabular}{ r|  c  c  c c }
$R_2'$ & 1 &2 &3 & 0\\
  \hline			
1& 3 & 0 & 1 & 2 \\
2& 3 & 2 & 1 & 0 \\
3& 3 & 0 & 1 & 2 \\
0& 3 & 2 & 1 & 0 \\
\end{tabular}.
\]
The operations $ \triangleright, R'_1$ and $R'_2$ have respectively the following $r^i(x)$ and $c^i(x)$ values for $i= 1,2,3$:
\[
\begin{tabular}{ r|  c  c  }
$x$ & $r^1(x)$ & $c^1(x)$\\
  \hline			
1&  2 & 4 \\
2&  4 & 2  \\
3&  2 & 4 \\
0&  4 & 2  \\
\end{tabular}
\qquad
\begin{tabular}{ r|  c  c  }
$x$ & $r^2(x)$ & $c^2(x)$\\
  \hline			
1&  0 & 0 \\
2&  2 & 2  \\
3&  0 & 0 \\
0&  2 & 2  \\
\end{tabular}
\qquad
\begin{tabular}{ r|  c  c  }
$x$ & $r^3(x)$ & $c^3(x)$\\
  \hline			
1&  1 & 1 \\
2&  1 & 1  \\
3&  1 & 1 \\
0&  1 & 1  \\
\end{tabular}.
\]
Thus, the singquandle polynomial of $Y$ is 
\[ sqp(Y) = 2 s_1^2 t_1^4 s_3 t_3+2 s_1^4 t_1^2 s_2^2 t_2^2 s_3 t_3. \]
Therefore, by Proposition~\ref{singinv} we can distinguish the singquandle $(Y,\triangleright, R_1',R_2')$ in this example and the singquandle $(X,*,R_1,R_2)$ in Example~\ref{sqp}.
\end{example}

We will now extend the subquandle polynomial defined in \cite{N} to the case of subsingquandles.

\begin{definition}
Let $X$ be a finite singquandle and $S \subset  Q$ a subsingquandle. Then the \emph{subsingquandle polynomial} is 
\[ Ssqp(S \subset X ) = \sum_{x \in S} s_1^{r^1(x)}t_1^{c^1(x)}s_2^{r^2(x)}t_2^{c^2(x)}s_3^{r^3(x)}t_3^{c^3(x)}. \]
\end{definition} 
Therefore, the subsingquandle polynomials are the contributions to the singquandle polynomial coming from the subsingquandles we are considering.

\begin{example}
Consider the subsingquandle $S_1=\lbrace 1,3 \rbrace$ of $X$ defined in Example \ref{sqp}. Thus, the subsingquandle polynomial 
\[Ssqp(S_1 \subset X) = 2 s_1^2 t_1^2 s_2 t_2 s_3^4   t_3^4. \]
\end{example}

For any singular link $L$, there is an associated fundamental singquandle $(\mathcal{SQ}(L), *, R_1, R_2)$, and for a finite singquandle $(T, \triangleright, R_1', R_2')$ the set of singquandle homomorphisms
\begin{equation*}
\begin{split}
\textup{Hom}(\mathcal{SQ}(L),T) = \lbrace f \, : \,  &\mathcal{SQ}(L) \rightarrow T \, \vert \, f( x*y) = f(x) \triangleright f(y),\\ &f(R_1(x,y))= R_1'(f(x),f(y)),f(R_2(x,y))= R_2'(f(x),f(y)) \rbrace     
\end{split}
\end{equation*}
can be used to construct computable invariants for singular knots. For example, by computing the cardinality of this set we obtain the singquandle counting invariant defined in \cite{CEHN}. By Lemma~\ref{image}, the image of a homomorphism from the fundamental singquandle of $L$ to a singquandle $T$ is a subsingquandle of $T$. Therefore, we can modify the singquandle counting invariant to obtain a multiset-valued link invariant by considering the subsingquandle polynomial of the image of $f$, for each $f \in \textup{Hom}(\mathcal{SQ}(L), T)$. By computing the cardinality of this multiset we retrieve the singquandle counting invariant. Additionally, the subsingquandle polynomial is a generalization of the specialized subquandle polynomial invariant which was shown in \cite{N} to distinguish some classical links which have the same quandle counting invariant.

\begin{definition}
Let $L$ be a singular link, $T$ be a finite singquandle. Then the multiset
\[ \Phi_{Ssqp}(L,T) = \lbrace Ssqp(Im(f) \subset T) \, \vert \, f \in \text{Hom}(\mathcal{SQ}(L), T )\rbrace \]
is the \emph{subsingquandle polynomial invariant of $L$} with respect to $T$. We can rewrite the multiset in a polynomial-style form by converting the multiset elements to exponents of a formal variable $u$ and converting their multiplicities to coefficients:
\[ \phi_{Ssqp}(L,T) = \sum_{f \in \textup{Hom}(\mathcal{SQ}(L),T)} u^{Ssqp(Im(f) \subset T)}.\]
\end{definition}
\begin{remark}
Let $t_i=s_i=0$ for $i=1,2,3$ in the subquandle polynomial invariant of $L$, then we obtain
\[  \phi_{Ssqp}(L,T) = \sum_{f \in \textup{Hom}(\mathcal{SQ}(L),T)} u^0 = \vert \textup{Hom}(\mathcal{SQ}(L),T)\vert.\]
We will show that the subsingquandle polynomial invariant contains more data than just the singquandle counting invariant.
\end{remark}
\section{Examples}\label{examples}
In this section, we distinguish singular knots and links by computing the subsingquandle polynomial invariant defined in the previous section. 

\begin{example}
Consider the 2-bouquet graphs of type $L$ listed as $1^l_1$ and $6_{11}^l$ in \cite{Oyamaguchi}. Let $(X,*,R_1,R_2)$ be the singquandle with $X = \mathbb{Z}_8$ and operations $x*y=7 x + 6 y + 4 x y=x\bar{*}y$, $R_1(x,y) = 2 x + 7 y + 4 x y$ and $R_2(x,y)=4 x^2+5 x+4 y$. A direct computation shows that the operations $*, R_1$ and $R_2$ satisfy the axioms of Definition~\ref{oriented SingQdle}. We can also represent these operations by the following operation tables,
\[
\begin{array}{l|cccccccc}
* & 1 & 2 & 3 & 4 & 5 & 6 & 7 & 0\\
\hline
1 & 1 & 3 & 5 & 7 & 1 & 3 & 5 & 7 \\
2 & 4 & 2 & 0 & 6 & 4 & 2 & 0 & 6 \\
3 & 7 & 1 & 3 & 5 & 7 & 1 & 3 & 5 \\
4 & 2 & 0 & 6 & 4 & 2 & 0 & 6 & 4 \\
5 & 5 & 7 & 1 & 3 & 5 & 7 & 1 & 3 \\
6 & 0 & 6 & 4 & 2 & 0 & 6 & 4 & 2 \\
7 & 3 & 5 & 7 & 1 & 3 & 5 & 7 & 1 \\
0 & 6 & 4 & 2 & 0 & 6 & 4 & 2 & 0 \\
\end{array}
\hspace{.25in}
\begin{array}{l|cccccccc}
R_1 & 1 & 2 & 3 & 4 & 5 & 6 & 7 & 0\\
\hline
1& 5 & 0 & 3 & 6 & 1 & 4 & 7 & 2 \\
2& 3 & 2 & 1 & 0 & 7 & 6 & 5 & 4 \\
3& 1 & 4 & 7 & 2 & 5 & 0 & 3 & 6 \\
4& 7 & 6 & 5 & 4 & 3 & 2 & 1 & 0 \\
5& 5 & 0 & 3 & 6 & 1 & 4 & 7 & 2 \\
6& 3 & 2 & 1 & 0 & 7 & 6 & 5 & 4 \\
7& 1 & 4 & 7 & 2 & 5 & 0 & 3 & 6 \\
0& 7 & 6 & 5 & 4 & 3 & 2 & 1 & 0 \\
\end{array}
\hspace{.25in}
\begin{array}{l|cccccccc}
R_2 & 1 & 2 & 3 & 4 & 5 & 6 & 7 & 0\\
\hline
1& 5 & 1 & 5 & 1 & 5 & 1 & 5 & 1 \\
2& 6 & 2 & 6 & 2 & 6 & 2 & 6 & 2 \\
3& 7 & 3 & 7 & 3 & 7 & 3 & 7 & 3 \\
4& 0 & 4 & 0 & 4 & 0 & 4 & 0 & 4 \\
5& 1 & 5 & 1 & 5 & 1 & 5 & 1 & 5 \\
6& 2 & 6 & 2 & 6 & 2 & 6 & 2 & 6 \\
7& 3 & 7 & 3 & 7 & 3 & 7 & 3 & 7 \\
0& 4 & 0 & 4 & 0 & 4 & 0 & 4 & 0 \\
\end{array}.
\]
The following tables include $r^i(x)$ and $c^i(x)$ values for $i=1,2,3$,
\[
\begin{array}{l|cc}
x & r^1(x) & c^1(x) \\
\hline
1& 2 & 2 \\
2& 2 & 2 \\
3& 2 & 2 \\
4& 2 & 2 \\
5& 2 & 2 \\
6& 2 & 2 \\
7& 2 & 2 \\
0& 2 & 2 \\
\end{array}
\hspace{.5in}
\begin{array}{l|cc}
x & r^2(x) & c^2(x) \\
\hline
1& 1 & 1 \\
2& 1 & 1 \\
3& 1 & 1 \\
4& 1 & 1 \\
5& 1 & 1 \\
6& 1 & 1 \\
7& 1 & 1 \\
0& 1 & 1 \\
\end{array}
\hspace{.5in}
\begin{array}{l|cc}
x & r^3(x) & c^3(x)\\
\hline
1 & 4 & 0 \\
2 & 4 & 8 \\
3 & 4 & 0 \\
4 & 4 & 8 \\
5 & 4 & 0 \\
6 & 4 & 8 \\
7 & 4 & 0 \\
0 & 4 & 8 \\
\end{array}.
\]
The singular link $1_1^l$, Figure~\ref{11l}, 
\begin{figure}[h]
\begin{tikzpicture}[use Hobby shortcut,scale=1.7]
%diagram on the left
\begin{knot}[
clip width=5
%  draft mode=crossings,
 % flip crossing=4
]

\strand[decoration={markings,mark=at position .5 with
    {\arrow[scale=3,>=stealth]{<}}},postaction={decorate}](-1,0) circle[radius=2cm];
\strand[decoration={markings,mark=at position .5 with
    {\arrow[scale=3,>=stealth]{>}}},postaction={decorate}] (1,0) circle[radius=2cm];

\end{knot}

\node[circle,draw=black, fill=black, inner sep=0pt,minimum size=12pt] (a) at (0,1.73) {};
\node[left] at (-1,0) {\tiny $z$};
%\node[right] at (1,0) {\tiny $R_2(x,y)$};
\node[left] at (-2,2) {\tiny $x$};
\node[right] at (2,2) {\tiny $y$};
\end{tikzpicture}
\vspace{.2in}
		\caption{Diagram of $1_1^l$.}
		\label{11l}
\end{figure}
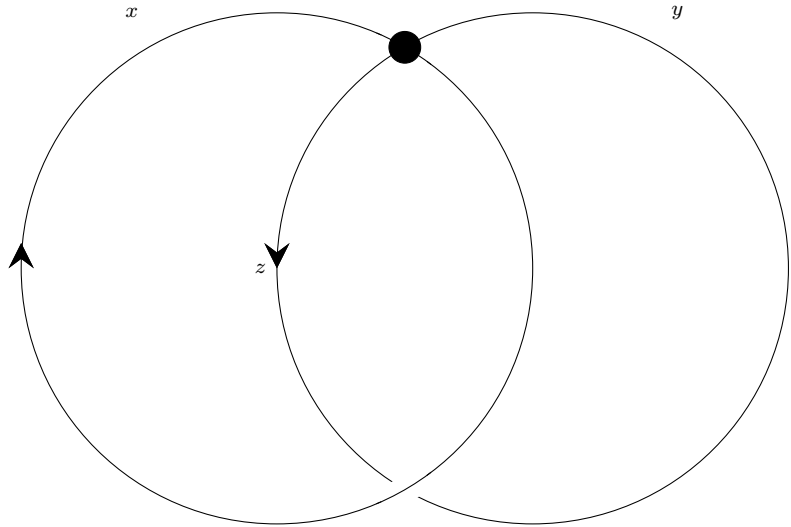
has the system of coloring equations given by 
\begin{eqnarray*}
4 x^2 + 5 x +  4 y &=&R_2(x,y) = x \\
7 z + 6 x + 4 x z &=& z * x = y\\
2 x + 7 y + 4 x y &=& R_1(x,y) = z.
\end{eqnarray*}
Thus, $1_1^l$ has 16 coloring by $X$. In Table~\ref{list1}, we list all coloring of $1_1^l$ by $X$ with the corresponding subsingquandle. 
\begin{table}
\centering
\begin{tabular}{ccc|c}
$x$ & $y$ & $z$ &  subsingquandle\\
\hline
 2 & 2 & 2 & $\{ 2 \}$ \\
 2 & 4 & 0 & $\{ 2,4,6,0 \}$\\
 2 & 6 & 6 &  $\{ 2,6 \}$\\
 2 & 0 & 4 & $\{ 2,4,6,0 \}$\\
 4 & 2 & 6 &  $\{ 2,4,6,0 \}$\\
 4 & 4 & 4 & $\{ 4 \}$ \\
 4 & 6 & 2 & $\{ 2,4,6,0 \}$\\
 4 & 0 & 0 &  $\{ 4,0 \}$\\
 6 & 2 & 2 & $\{ 2,6 \}$\\
 6 & 4 & 0 & $\{ 2,4,6,0 \}$\\
 6 & 6 & 6 &  $\{ 6 \}$\\
 6 & 0 & 4 & $\{ 2,4,6,0 \}$\\
 0 & 2 & 6 & $\{ 2,4,6,0 \}$\\
 0 & 4 & 4 & $\{ 4,0\}$\\
 0 & 6 & 2 & $\{ 2,4,6,0 \}$\\
 0 & 0 & 0 & $\{ 0 \}$ \\
\end{tabular}
\caption{\label{list1} List of colorings of $1_1^l$ by $X$ and the corresponding subsingquandle.}
\end{table}

The singular link $6_{11}^l$, Figure~\ref{611l},
\begin{figure}[h]
\begin{tikzpicture}[use Hobby shortcut, scale=1.3]
%diagram on the left
\begin{knot}[
 consider self intersections=true,
 clip width=5,
 % draft mode=crossings,
 flip crossing/.list={4,5,7}
]

\strand ([closed]-3,0)[decoration={markings,mark=at position .5 with
    {\arrow[scale=3,>=stealth]{>}}},postaction={decorate}]..(0,-3)..(3,0)..(0,3);
\strand ([closed]0,-.5)[decoration={markings,mark=at position .5 with
    {\arrow[scale=3,>=stealth]{>}}},postaction={decorate}]..(-4,-3)..(-1.5,-3)..(1.5,0)..(0,1.5)..(-1.5,0)..(1.5,-3)..(4,-3)..(0,-.5);

\end{knot}

\node[circle,draw=black, fill=black, inner sep=0pt,minimum size=12pt] (a) at (-2.85,-1) {};
\node[left] at (-4.2,-2) {\tiny $R_1(x,y)$};
\node[left] at (-2.3,-2) {\tiny $R_2(x,y)$};
\node[left] at (-3,0) {\tiny $x$};
\node[above] at (-2,-.8) {\tiny $y$};
\node[right] at (4.2,-2) {\tiny $z$};
\node[above] at (.6,-2.6) {\tiny $k$};
\node[above] at (0,1.5) {\tiny $w$};
\draw [->,>=stealth] (5,0) -- (1.3,-1.2);
\node[above] at (5,0) {\tiny $R_1(x,y) * R_2(x,y)$};
\end{tikzpicture}
\vspace{.2in}
		\caption{Diagram of $6_{11}^l$.}
		\label{611l}
\end{figure}
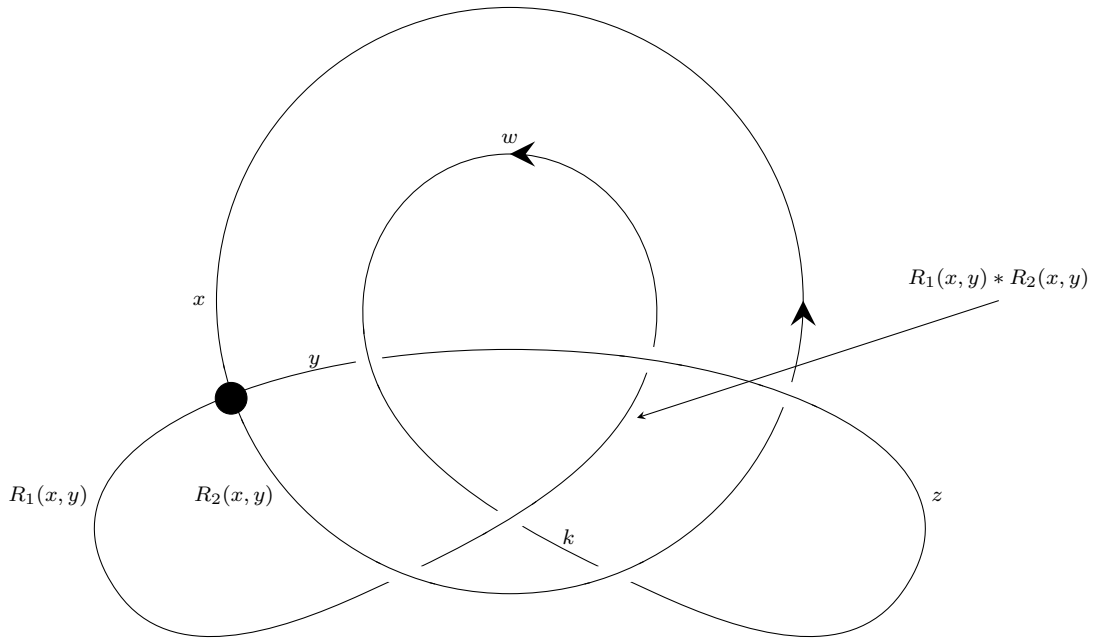
has the system of coloring equations given by 
\begin{eqnarray*}
3 x + 4 x^2 + 4 y + 6 z + 4 x z &=&R_2(x,y) \bar{*} z = x \\
6 w + 7 z + 4 w z &=&z \bar{*}w = y\\
7 k + 6 x + 4 k x &=&k \bar{*}R_2(x,y) = z\\
4 x + 7 y + 6 z + 4 y z &=&(R_1(x,y) * R_2(x,y))\bar{*}z = w\\
7 w + 6 y + 4 w y &=& w \bar{*} (R_1(x,y) * R_2(x,y)) =k.
\end{eqnarray*}
Therefore, $6_{11}^l$ also has 16 coloring by $X$. In Table~\ref{list2},  we list all coloring of $1_1^l$ by $X$ with the corresponding subsingquandle.

\begin{table}[h]
    \centering
\begin{tabular}{cccccccc|c}
 $x$ & $y$ & $z$ & $w$ & $k$ & $R_1(x,y)$ & $R_2(x,y)$ & $R_1(x,y) * R_2(x,y)$ & subsingquandle \\
 \hline
 1 & 3 & 3 & 7 & 7 & 3 & 5 & 7 & $\{1,3,5,7 \}$\\
 1 & 7 & 7 & 3 & 3 & 7 & 5 & 3 & $\{1,3,5,7 \}$\\
 2 & 2 & 2 & 2 & 2 & 2 & 2 & 2 &  $\{ 2\}$\\
 2 & 6 & 6 & 6 & 6 & 6 & 2 & 6 & $\{ 2,6\}$\\
 3 & 1 & 1 & 5 & 5 & 1 & 7 & 5 & $\{1,3,5,7 \}$\\
 3 & 5 & 5 & 1 & 1 & 5 & 7 & 1 & $\{1,3,5,7 \}$\\
 4 & 4 & 4 & 4 & 4 & 4 & 4 & 4 &  $\{ 4\}$\\
 4 & 0 & 0 & 0 & 0 & 0 & 4 & 0 & $\{ 4,0\}$\\
 5 & 3 & 3 & 7 & 7 & 3 & 1 & 7 & $\{1,3,5,7 \}$\\
 5 & 7 & 7 & 3 & 3 & 7 & 1 & 3 & $\{1,3,5,7 \}$\\
 6 & 2 & 2 & 2 & 2 & 2 & 6 & 2 & $\{ 2,6\}$\\
 6 & 6 & 6 & 6 & 6 & 6 & 6 & 6 & $\{ 6\}$\\
 7 & 1 & 1 & 5 & 5 & 1 & 3 & 5 & $\{1,3,5,7 \}$\\
 7 & 5 & 5 & 1 & 1 & 5 & 3 & 1 & $\{1,3,5,7 \}$\\
 0 & 4 & 4 & 4 & 4 & 4 & 0 & 4 &  $\{ 4,0\}$\\
  0 & 0 & 0 & 0 & 0 & 0 & 0 & 0 & $\{ 0\}$
\end{tabular}
\caption{\label{list2} List of colorings of $6_1^l$ by $X$ and the corresponding subsingquandle.}
\end{table}

Furthermore, both $1^l_1$ and $6_{11}^l$ have the same Jablan polynomial, $\Delta_J(1^l_1)=s-t=\Delta_J(6^l_{11})$ as computed in \cite{NOS}. However, the subsingquandle polynomial invariants
\[ \phi_{Ssqp}(1_1^l,X) = 8 u^{4 s_1^2t_1^2 s_2  t_2 s_3^4  t_3^8}+ 4 u^{2 s_1^2 t_1^2 s_2 t_2 s_3^4 t_3^8}+4 u^{s_1^2 t_1^2 s_2 t_2 s_3^4 t_3^8}\] 
and
\[\phi_{Ssqp}(6_{11}^l,X) =8 u^{4 s_1^2 t_1^2 s_2 t_2 s_3^4 }+4
   u^{2 s_1^2 t_1^2 s_2 t_2 s_3^4 t_3^8}+4 u^{s_1^2 t_1^2 s_2 t_2 s_3^4 t_3^8}\]
distinguish these singular links.
\end{example}

\begin{example}
Consider the singquandle $X = \mathbb{Z}_8$ with operations $x*y = 5x+4y = x\bar{*} y$ and $R_1(x,y)= 6+5x+6xy$, and $R_2(x,y)= 6+5y+6xy$. A direct computation shows that the operations $*, R_1$ and $R_2$ satisfy the axioms of Definition~\ref{oriented SingQdle}.
%We can also represent the operations by the following operation tables
%\[
%\begin{array}{l|cccccccc}
%* &1 & 2 & 3 & 4 & 5 & 6 & 7 & 0\\
%\hline
%1 & 1 & 5 & 1 & 5 & 1 & 5 & 1 & 5 \\
%2 & 6 & 2 & 6 & 2 & 6 & 2 & 6 & 2 \\
%3 & 3 & 7 & 3 & 7 & 3 & 7 & 3 & 7 \\
%4 & 0 & 4 & 0 & 4 & 0 & 4 & 0 & 4 \\
%5 & 5 & 1 & 5 & 1 & 5 & 1 & 5 & 1 \\
%6 & 2 & 6 & 2 & 6 & 2 & 6 & 2 & 6 \\
%7 & 7 & 3 & 7 & 3 & 7 & 3 & 7 & 3 \\
%0 & 4 & 0 & 4 & 0 & 4 & 0 & 4 & 0 \\
%\end{array}
%\qquad
%\begin{array}{l|cccccccc}
%R_1 &1 & 2 & 3 & 4 & 5 & 6 & 7 & 0\\
%\hline
%1& 1 & 7 & 5 & 3 & 1 & 7 & 5 & 3 \\
%2& 4 & 0 & 4 & 0 & 4 & 0 & 4 & 0 \\
%3& 7 & 1 & 3 & 5 & 7 & 1 & 3 & 5 \\
%4& 2 & 2 & 2 & 2 & 2 & 2 & 2 & 2 \\
%5& 5 & 3 & 1 & 7 & 5 & 3 & 1 & 7 \\
%6& 0 & 4 & 0 & 4 & 0 & 4 & 0 & 4 \\
%7& 3 & 5 & 7 & 1 & 3 & 5 & 7 & 1 \\
%0& 6 & 6 & 6 & 6 & 6 & 6 & 6 & 6 \\
%\end{array}
%\qquad
%\begin{array}{l|cccccccc}
%R_2 &1 & 2 & 3 & 4 & 5 & 6 & 7 & 0\\
%\hline
%1& 1 & 4 & 7 & 2 & 5 & 0 & 3 & 6 \\
%2& 7 & 0 & 1 & 2 & 3 & 4 & 5 & 6 \\
%3& 5 & 4 & 3 & 2 & 1 & 0 & 7 & 6 \\
%4& 3 & 0 & 5 & 2 & 7 & 4 & 1 & 6 \\
%5& 1 & 4 & 7 & 2 & 5 & 0 & 3 & 6 \\
%6& 7 & 0 & 1 & 2 & 3 & 4 & 5 & 6 \\
%7& 5 & 4 & 3 & 2 & 1 & 0 & 7 & 6 \\
%0& 3 & 0 & 5 & 2 & 7 & 4 & 1 & 6 \\
%\end{array}.
%\]
The following singular knot $K_1$ below has the following coloring equations
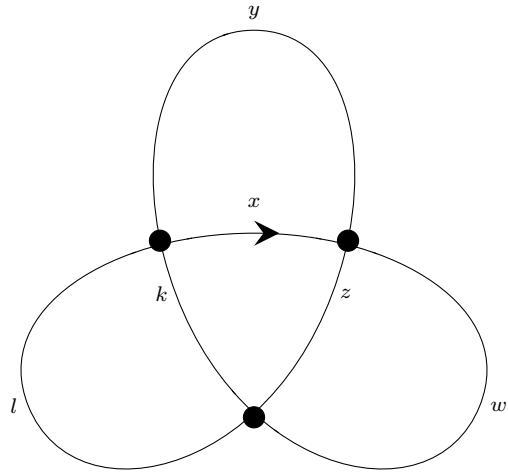
\begin{figure}[h]
\begin{tikzpicture}[use Hobby shortcut]
\begin{knot}[
  consider self intersections=true,
%  draft mode=crossings,
  ignore endpoint intersections=false,
  flip crossing=3,
  only when rendering/.style={
%    show curve endpoints
  }
  ]
\strand ([closed]0,1.5)[decoration={markings,mark=at position .5 with
    {\arrow[scale=3,>=stealth]{<}}},postaction={decorate}]..(-1.2,-1.5).. (3,-3.5) ..(0,-1.2) ..(-3,-3.5) ..(1.2,-1.5)..(0,1.5);
\end{knot}
\node[circle,draw=black, fill=black, inner sep=0pt,minimum size=8pt] (a) at (0,-3.65) {};
\node[circle,draw=black, fill=black, inner sep=0pt,minimum size=8pt] (a) at (1.25,-1.3) {};
\node[circle,draw=black, fill=black, inner sep=0pt,minimum size=8pt] (a) at (-1.25,-1.3) {};
\node[above] at (0,-1) {\tiny $x$};
\node[above] at (0,1.5) {\tiny $y$};
\node[right] at (1,-2) {\tiny $z$};
\node[right] at (3,-3.5) {\tiny $w$};
\node[left] at (-3,-3.5) {\tiny $l$};
\node[left] at (-1,-2) {\tiny $k$};

\end{tikzpicture}
%\vspace{.2in}
		\caption{Diagram of $K_1$.}
		\label{k1}
\end{figure}
%has the following coloring equations
\begin{eqnarray*}
6 + 5 k + 6 k l= R_1(k,l) &=& x \\
6 + 5 l + 6 k l= R_2(k,l) &=& y \\
6 + 5 x + 6 x y= R_1(x,y) &=& z \\
6 + 5 y + 6 x y= R_2(x,y) &=& w \\
6 + 5 w + 6 w z= R_2(z,w) &=& l \\
6 + 5 z + 6 w z= R_1(z,w) &=& k.
\end{eqnarray*}
Therefore, $K_1$ has 8 colorings by $X$. The singular knot $K_2$ below
\begin{figure}[h]
\begin{tikzpicture}[use Hobby shortcut]
\begin{knot}[
  consider self intersections=true,
%  draft mode=crossings,
  ignore endpoint intersections=false,
  flip crossing=4,
  clip width=5,
  only when rendering/.style={
%    show curve endpoints
  }
  ]
\strand ([closed]0,1.5)[decoration={markings,mark=at position .5 with
    {\arrow[scale=3,>=stealth]{<}}},postaction={decorate}]..(-1.2,-1.5).. (3,-3.5) ..(0,-1.2) ..(-3,-3.5) ..(1.2,-1.5)..(0,1.5);
\end{knot}
\node[circle,draw=black, fill=black, inner sep=0pt,minimum size=9.5pt] (a) at (1.25,-1.35) {};
\node[circle,draw=black, fill=black, inner sep=0pt,minimum size=9.6pt] (a) at (-1.241,-1.35) {};
\node[above] at (0,-1) {\tiny $x$};
\node[above] at (0,1.5) {\tiny $y$};
\node[below] at (-1.5,-4.5) {\tiny $z$};
\node[right] at (3,-3.5) {\tiny $w$};
%\node[left] at (-3,-3.5) {\tiny $l$};
\node[left] at (-1,-2.5) {\tiny $k$};
\end{tikzpicture}
%\vspace{.2in}
		\caption{Diagram of $K_2$.}
		\label{K2}
\end{figure}
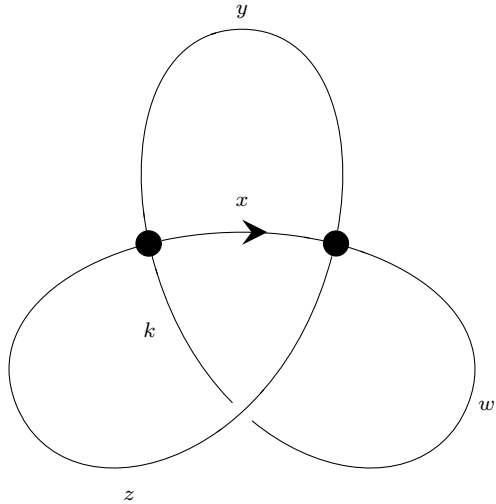
has the following coloring equations
\begin{eqnarray*}
6 + 5 k + 6 k z= R_1(k,z) &=& x\\
6 + 5 z + 6 k z= R_2(k,z) &=& y\\
6 + 5 x + 6 x y= R_1(x,y) &=& z\\
6 + 5 y + 6 x y= R_2(x,y) &=& w\\
5 w + 4 z= w \bar{*}z &=& k.
\end{eqnarray*}
Therefore, $K_2$ also has 8 colorings by $X$. However, the subsingquandle polynomial  invariants
\[ \phi_{Ssqp}(K_1,X) = 4 u^{s_1^4 t_1^4 s_2^2 t_2^2 s_3 t_3}+4 u^{2 s_1^4 t_1^4 s_2^2 t_2^2 s_3
   t_3}\]
and
\[ \phi_{Ssqp}(K_2,X)=4 u^{4 s_1^4 t_1^4 s_3 t_3}+4 u^{s_1^4 t_1^4 s_2^2 t_2^2 s_3 t_3}\]
distinguish the singular knots $K_1$ and $K_2$.
\end{example}

{\bf Acknowledgement}: The authors would like to thank the anonymous referee for fruitful comments which improved the paper.

\end{document}